\documentclass[12pt]{article}

\usepackage{a4wide} 
\usepackage{amssymb,amsmath,array}

\input xy
\xyoption{matrix}\xyoption{arrow}
\def\edge{\ar@{-}}
\def\dedge{\ar@{.}}

\newcommand{\qed}{\hfill $\square$}

\newtheorem{theorem}{Theorem}[section]
\newtheorem{proposition}[theorem]{Proposition}

\newtheorem{definition}[theorem]{Definition}
\newtheorem{lemma}[theorem]{Lemma}
\newtheorem{corollary}[theorem]{Corollary}

\newtheorem{example}[theorem]{Example}

\def\widebar{\overline}

\def\ch{{\mathcal H}}
\def\spec{{\rm Spec}}
\def\hspec{\ch\!-\!\spec} 

\def\mz{{\mathbb Z}}

\newcommand{\qs}[1]{\left[#1\right]}

\newcommand{\ireduced}
{\widetilde{i}, \widetilde{i+1},\dots ,\widetilde{i+m-1}}

\newcommand\curlyma{\{\widetilde{\alpha},\widetilde{\alpha+1},\dots,\widetilde{\alpha+m-1}\}}

\newcommand{\q}{\medskip\par\noindent}

\newcommand\content[1]{{\rm content}(#1)}

\def\oq{\mathcal{O}_{q}}

\def\oqmmn{\oq(M_{m,n})}

\def\oqgmn{\oq(G(m,n))}

\def\oqgmdashn{\oq(G(m',n))}
\def\oqgtwofour{\oq(G(2,4))}
\def\gmn{G(m,n)}
\def\tnngmn{\gmn^{{\rm tnn}}}

\def\k{{\mathbb K}}
\def\gk{{\rm GKdim}}

\def\goesto{\longrightarrow}

\def\goesto{\longrightarrow} 

\title{Twisting the quantum grassmannian
}

\author{S Launois
\thanks{\,The research of the first named author was
supported by a Marie Curie European Reintegration Grant within the
$7^{\mbox{th}}$ European Community Framework Programme.}
~~and T H Lenagan
}

\date{}

\begin{document}
\maketitle 

\abstract{\footnotesize In contrast to the classical and semiclassical
settings, the Coxeter element $(12\dots n)$ which cycles the columns of an 
$m\times n$ matrix does not determine an automorphism of the quantum
grassmannian. Here, we show that this cycling can be obtained by defining a
cocycle twist. A consequence is that the torus invariant prime ideals of the
quantum grassmannian are permuted by the action of the Coxeter element
$(12\dots n)$; we view this as a quantum analogue of the recent result of
Knutson, Lam and Speyer that the Lusztig strata of the classical grassmannian
are permuted by $(12\dots n)$.}

\vskip .5cm
\noindent
{\em 2000 Mathematics subject classification:} 16W35, 16P40, 16S38, 17B37,
20G42. 

\vskip .5cm
\noindent
{\em Key words:} 
Quantum matrices, quantum grassmannian, cocycle twist, 
noncommutative dehomogenisation 

\section{Introduction} 

The symmetric group $S_n$ acts on the grassmannian $G(m,n)$ by permuting the
columns of an $m\times n$ matrix that determines a point in $G(m,n)$. If one
restricts to considering the totally nonnegative grassmannian $\tnngmn$ this
is no longer true; however, Postnikov, \cite[Remark 3.3]{post}, notes that the
cycle $c=(12\dots n)$ acts on the totally nonnegative grassmannian. Recently,
Knutson, Lam and Speyer, \cite{knut-lam-speyer}, showed that the Lusztig
strata of the classical grassmannian are permuted by $(12\dots n)$. In fact
this invariance property is even stronger. Indeed, Goodearl and Yakimov,
\cite{good-yak}, have found a Poisson interpretation of the Lusztig strata:
they coincide with the $\ch$-orbits of symplectic leaves of $G(m,n)$, where
$\ch$ is an $n$-dimensional algebraic torus. Recently Yakimov, \cite{yak},
showed that the Coxeter element $c$ induces a Poisson automorphism of
$G(m,n)$. As a consequence he showed that the $\ch$-orbits of symplectic
leaves of $G(m,n)$ are permuted by $c$; this gives a Poisson geometric proof
of Knutson, Lam and Speyer result.

 In view of the
close connections that have been discovered between totally nonnegative
matrices, the standard Poisson matrix variety and quantum matrices, see, for
example, \cite{good-laun-len, good-laun-len2}, and between the totally
nonnegative grassmannian and the quantum grassmannian, see, for example,
\cite{llr}, one might expect that the cycle $c$ produces an automorphism of
the quantum grassmannian. This is not the case, see
Example~\ref{example-nonautomorphism} below. With this in mind, one wonders
what the analogous result should be. Here, we provide the answer: there is a
$2$-cocyle which can be used to twist the quantum grassmannian; the resulting
twisted algebra is again isomorphic to the quantum grassmannian, and the
effect of the twist on a generating quantum minor $I$ is to produce (a scalar
multiple of) the quantum minor obtained by letting the cycle $c$ act on the
indices of $I$. A consequence of this result is that the torus invariant prime
ideals of the quantum grassmannian are permuted by the cycle $(12\dots n)$,
see Corollary~\ref{corollary-primes-permuted}; we view this as a quantum
analogue of the Knutson, Lam and Speyer result. \\

\section{Basic definitions}\label{section-basicdefs}

In this section, we will give the basic definitions of the objects that
interest us in this paper and recall several results that we need in our
proofs. 
Throughout, $\mathbb{K}$ will denote a base field, we set $\mathbb{K}^*:= \mathbb{K} \setminus \{0\}$, $q$ will be a
non-zero element of $\mathbb{K}$ and $m$ and $n$ denote positive integers with $m<n$. Moreover, we assume that there exists $p \in \mathbb{K}$ such that $p^m=q^2$. \\

\q
The quantisation of the coordinate ring of the affine variety $M_{m,n}$ of
$m \times n$ matrices with entries in $\mathbb{K}$ is denoted ${\mathcal
O}_q(M_{m,n})$. It is the $\mathbb{K}$-algebra generated by $mn$ indeterminates
$X_{ij}$, with $1 \le i \le m$ and $1 \le j \le n$, subject to the relations:
\[
\begin{array}{ll}  
X_{ij}X_{il}=qX_{il}X_{ij},&\mbox{ for }1\le i \le m,\mbox{ and }1\le j<l\le
n\: ;\\ 
X_{ij}X_{kj}=qX_{kj}X_{ij}, & \mbox{ for }1\le i<k \le m, \mbox{ and }
1\le j \le n \: ; \\ 
X_{ij}X_{kl}=X_{kl}X_{ij}, & \mbox{ for } 1\le k<i \le m,
\mbox{ and } 1\le j<l \le n \: ; \\
X_{ij}X_{kl}-X_{kl}X_{ij}=(q-q^{-1})X_{il}X_{kj}, & \mbox{ for } 1\le i<k \le
m, \mbox{ and } 1\le j<l \le n.
\end{array}
\]

\q
An {\it index pair} is a pair $(I,J)$ such
that $I \subseteq \{1,\dots,m\}$ and $J \subseteq \{1,\dots,n\}$ are subsets
with the same cardinality. Hence, an index pair is given by an integer $t$
such that $1 \le t \le m$ and ordered sets 
$I=\{i_1 < \dots < i_t\} \subseteq \{1,\dots,m\}$
and $J=\{j_1 < \dots < j_t\} \subseteq \{1,\dots,n\}$. To any such index pair
we associate the quantum minor 
\[ 
[I|J] = \sum_{\sigma\in S_t}
(-q)^{\ell(\sigma)} X_{i_{\sigma(1)}j_1} \cdots X_{i_{\sigma(t)}j_t} . 
\] 


\begin{definition}\label{def-q-grassmannian} 
{\rm 
The {\it quantisation of the coordinate ring of the grassmannian of
$m$-dimensional subspaces of $\mathbb{K}^n$}, denoted by $\oqgmn$
and informally referred to as the ($m\times n$) {\em quantum grassmannian} is the
subalgebra of $\oqmmn$ generated by the $m \times m$
quantum minors.
}
\end{definition} 


\q
A maximal (that is, $m\times m$) quantum minor in $\oqmmn$ corresponds to an
index pair $[\{1,\dots,m\}|J]$ with $J=\{j_1,\dots,j_m\}\subseteq
\{1,\dots,n\}$. We call such $J$ {\em index sets} and denote the corresponding
minor by $[J]$ or $[j_1,\dots,j_m]$ in what follows. Thus, such a $[J]$ is a generator of $\oqgmn$.

\q
When writing down an $m\times m$ quantum minor in $\oqgmn$, we will use the
convention that if a column index $j$ is greater than $n$ then $j$ is to be
read as $j-n$. For example, in $\oqgtwofour$ the minor specified by $[45]$ is
the quantum minor $[14]$. In order to stress
this point, we will use the convention that given any integer $j$ then
$\widetilde{j}$ is the integer in the set $\{1,\dots,n\}$ that is congruent to
$j$ modulo $n$. 

\q
A quantum minor $[\ireduced ]$ is said to be a {\em consecutive
quantum minor} of $\oqgmn$. Recalling the convention above, we see that there
are four consecutive minors in $\oqgtwofour$: they are $[12], [23], [34]$ and
$[\widetilde{4}\,\widetilde{5}] = [14]$. 
More generally, $\oqgmn$ has $n$ consecutive minors. 

\q
Two maximal quantum minors $[I]$ and $[J]$ are said to {\em quasi-commute} if
there is an integer $c$ such that $[I][J] = q^c[J][I]$. Recall that an element
$u$ of a ring $R$ is said to be a {\em normal element} if $uR = Ru$, in which
case $uR$ is a two-sided ideal. The following lemma, first obtained in
\cite[Lemma 3.7]{kroblec}, shows that consecutive quantum minors quasi-commute with all
maximal quantum minors.

\begin{lemma}\label{lemma-normal}

Let $[\ireduced ]$ be a consecutive quantum minor in the quantum grassmannian 
$\oqgmn$.
Then $[\ireduced ]$ quasi-commutes with each of the generating
quantum minors of $\oqgmn$. In particular, each  consecutive quantum minor 
is a normal
element of $\oqgmn$. \qed 
\end{lemma} 

\q
A consequence of this result is that the powers of a consecutive quantum minor
form an Ore set in the noetherian domain $\oqgmn$; and so it is possible to invert the consecutive
quantum minor in a localisation.

\q
In order to facilitate computations, we need a version of the Quantum Muir's
Law of Extensible Minors. This result was first obtained by Krob and Leclerc,
\cite[Theorem 3.4]{kroblec}, with a proof involving quasi-determinants. The
version below, which is sufficient for our needs, is taken from
\cite[Proposition 1.3]{lr3}, and is adapted for use in the quantum
grassmannian.

\begin{proposition} \label{proposition-qmuir} 
Let $I_s, J_s$, for $1\leq s\leq d$, be $m$-element subsets of $\{1,\dots,n\}$
and let $c_s\in\mathbb{K}$ be such that $\sum_{s=1}^d c_s[I_s][J_s]=0$ 
in $\oqgmn$. Suppose that
$P$ is a subset of $\{1,\dots,n\}$ such that $(\cup_{s=1}^d I_s) \cup
(\cup_{s=1}^d J_s)\subseteq P$ and let 
$\widebar{P}$ denote $\{1,\dots,n\}\backslash
P$. Then 
\[
\sum_{s=1}^d c_s[I_s\sqcup \widebar{P}][J_s\sqcup \widebar{P}]=0.
\]
holds in $\oqgmdashn$, where $m' = m+\#\widebar{P}$.\qed 
\end{proposition}

\q
This result is used, for example, 
when it is necessary to write down a commutation relation
between two maximal quantum minors $[I]$ and $[J]$, say. The usefulness of the
result is that one may delete the common members of the index pairs $I$ and
$J$ to establish the commutation relation.

\section{Cycling does not induce an automorphism}

In contrast to the classical and semiclassical settings, the cycle $(12\dots
n)$ does not act as an automorphism on the quantum grassmannian. We show this
here by considering $\oqgtwofour$. 

\q
First, we summarize the commutation relations and the quantum 
Pl\"ucker relation for 
$\oqgtwofour$; which can easily be obtained from the 
defining relations of quantum matrices.

\[[ij][ik] = q[ik][ij],\quad [ik][jk] = q[jk][ik],\quad
\mbox{\rm for $i<j<k$}\]
and
\[
\qs{14}\qs{23}  =  \qs{23}\qs{14}, \quad [12][34] =q^2[34][12], \quad
\qs{13}\qs{24}  =  \qs{24}\qs{13} 
        + \left( q-q^{-1} \right) \qs{14}\qs{23}.\]
        There is also a 
quantum Pl\"{u}cker relation
$
\qs{12}\qs{34} - q\qs{13}\qs{24} +q^2\qs{14}\qs{23} = 0
$. 
This  quantum Pl\"{u}cker relation may be rewritten as 
$
\qs{34}\qs{12} - q^{-1}\qs{24}\qs{13} +q^{-2}\qs{23}\qs{14} = 0
$ 
and one can also check that 
$
\qs{13}\qs{24}  = q^2 \qs{24}\qs{13} 
+ \left( q^{-1} -q\right) \qs{12}\qs{34}.
$

\begin{example}\label{example-nonautomorphism} 
{\rm 
Let $\theta[ij]:=[i+1,j+1]$, with the convention that $\theta(4)=1$; that is, 
we work modulo $4$ and $\theta$ is cycling the indices of quantum minors:
\[\theta[ij]=[\widetilde{c(i)},\widetilde{c(j)}],
 \]
where $c$ denotes the cycle $(1234)$.
\q
In the classical case, $\theta$ 
induces an isomorphism, and this is also the case in the Poisson setting, 
\cite{yak}. 

\q
However, $\theta$ does not induce an automorphism of $\oqgtwofour$, since, 
for example, the 
quantum Pl\"ucker relation is not preserved: if we assume that 
$\theta$ induces an automorphism then we calculate
\[
0=\theta(0) =
\theta( [12][34]-q[13][24]+q^2[14][23]) = 
[23][14]-q[24][13]+q^2[12][34].
\]
However, one can check that $[23][14]-q[24][13]+q^2[12][34]\neq 0$. For, suppose that
$[23][14]-q[24][13]+q^2[12][34]=0$, then $[23][14]-q[24][13]+q^4[34][12]=0$. However, from the
second version of the quantum Pl\"ucker relation, we know that
$[14][23]-q[24][13]+q^2[34][12]=0$. Subtract one of these equations from the other and
note that $[14][23]=[23][14]$ to obtain $(q^4-q^2)[34][12]=0$, a contradiction, 
provided that $q^2\neq 1$. 
}
\end{example}

\section{Dehomogenisation at a consecutive 
minor} 

Explicit calculations in the quantum grassmannian can be difficult due to 
the awkward defining relations (quantum Pl\"ucker relations). 
For this reason, 
it is often useful to transfer to an overring where the defining relations 
are simpler. This can be achieved by localising at any 
consecutive quantum minor, and this leads to consideration of the 
noncommutative dehomogenisation isomorphism for an arbitrary consecutive 
quantum minor. \\

\q 
Set $M_{\alpha}:=\curlyma$ in $\oqgmn$. Now, $[M_{\alpha}]$ is a normal element, by 
Lemma~\ref{lemma-normal}; and so we may form the localisation 
$\oqgmn[[M_{\alpha}]^{-1}]$. In $\oqgmn[[M_{\alpha}]^{-1}]$ 
set 
\[
x_{ij}:= [M_{\alpha}\cup
\{\widetilde{j+\alpha+m-1}\}\backslash \{\widetilde{\alpha+m-i}\}]
[M_{\alpha}]^{-1}.
\]

\begin{theorem} \label{theorem-dehomiso}
The subalgebra $\mathbb{K}[x_{ij}]$ of $\oqgmn[[M_{\alpha}]^{-1}]$ 
is a $q$-quantum matrix algebra; 
that is, $\mathbb{K}[x_{ij}]$ is isomorphic to $\mathcal{O}_q(M_{m,n-m})$ 
by an isomorphism that send $x_{ij}$ to $X_{ij}$. 
Moreover there is an isomorphism  
\[
\phi_{\alpha}:\oqgmn[[M_{\alpha}]^{-1}]\goesto 
\mathbb{K}[x_{ij}][y_{\alpha}^{\pm 1};\sigma_{\alpha}].
\]
where $\sigma_{\alpha}$ is the automorphism of the quantum matrix algebra 
$\mathbb{K}[x_{ij}]$ defined by 
$\sigma_{\alpha}(x_{ij})M_{\alpha} = M_{\alpha}x_{ij}$. 
Under this isomorphism, $y_{\alpha}= \phi_{\alpha}(M_{\alpha})$. 
\end{theorem} 

\begin{proof}
The fact that $\mathbb{K}[x_{ij}]$ is a quantum matrix algebra is established in
\cite[Theorem 3.2]{len-russ}. The inclusion $\rho_{\alpha}:\mathbb{K}[x_{ij}]\goesto
\oqgmn[[M_{\alpha}]^{-1}]$ extends to a homomorphism $\rho_{\alpha}:\mathbb{K}[x_{ij}][y_{\alpha}^{\pm
1};\sigma_{\alpha}]\goesto \oqgmn[[M_{\alpha}]^{-1}]$, by the universal property of skew 
polynomial extensions. The fact that the extension 
$\rho_{\alpha}$ is an isomorphism
follows from \cite[Lemma 3.1]{len-russ} and the dehomogenisation
isomorphism \cite[Lemma 3.1]{klr}. Now, set $\phi_{\alpha} = \rho_{\alpha}^{-1}$. 
\qed
\end{proof}

\q
Next, we need to calculate the effect of $\phi_{\alpha}$ on generating quantum minors
of $\oqgmn$. 

\q 
Let $I$ be an $m$-element subset of $\{1,\dots,n\}$. For a fixed $\alpha$, set 
$I_r:=I\cap M_{\alpha}$ and $I_c:= I\backslash I_r$; so that 
$I=I_r\sqcup I_c$ (the notation is chosen because $I_r$ will give information 
about the row set of the image of $[I]$ and $I_c$ will give information 
about the column set). 

\q 
To simplify the notation somewhat, 
if $N$ is a subset of integers, and $i$ is an integer, then
$i+N= \{i+k\mid k\in N\}$. 

\begin{corollary}\label{corollary-phi-of-I}
Let $[I]$ be a generating quantum minor of $\oqgmn$. Then 
\[
\phi_{\alpha}([I])= [(\alpha +m)- (M_{\alpha}\backslash I_r)\mid 
I_c-(\alpha +m-1)]y_{\alpha}.
\]
\end{corollary}

\begin{proof}
By using \cite[Proposition 4.3]{len-russ}, we see that for a 
quantum minor 
$[I|J]$ of the quantum matrix algebra $\mathbb{K}[x_{ij}]$ 
\[
\rho_{\alpha}([I|J])= 
[M_{\alpha}\backslash ((\alpha +m)-I)\sqcup 
((\alpha +m-1)+J)][M_{\alpha}]^{-1}.
\]
As $\phi_{\alpha} = \rho_{\alpha}^{-1}$,  
the claim will be established once we show that 
\[
\rho_{\alpha}([(\alpha +m)- (M_{\alpha}\backslash I_r)\mid 
I_c-(\alpha +m-1)] \cdot y_{\alpha})= [I].
\]
Now, 
\begin{eqnarray*}
\lefteqn{\rho_{\alpha}([(\alpha +m)- (M_{\alpha}\backslash I_r)\mid 
I_c-(\alpha +m-1)] \cdot y_{\alpha})
=}\\
&&
[M_{\alpha}\backslash ((\alpha +m)-((\alpha +m)-M_{\alpha}\backslash I_r)
\sqcup
((\alpha +m-1)+(I_c-(\alpha +m-1)))] 
[M_{\alpha}]^{-1} \cdot [M_{\alpha}]\\
&& =~~
[M_{\alpha}\backslash(M_{\alpha}\backslash I_r)\sqcup  I_c]= 
[I_r\sqcup  I_c]=[I],
\end{eqnarray*}
as required. 
\qed
\end{proof}

\q We shall need to use the isomorphisms $\phi_{\alpha}$ and $\rho_{\alpha}$ of
Theorem~\ref{theorem-dehomiso} in the two cases $\alpha=1$ and $\alpha=2$. The next two
results record the action of $\sigma_1$ and $\sigma_2$. 

\begin{lemma} \label{lemma-sigma-1}
For $1\leq i\leq m$ and $1\leq j\leq n-m$
\[
\sigma_1(x_{ij}) = qx_{ij}.
\] 
Consequently, $y_1x_{ij}=qx_{ij}y_1$ 
for $1\leq i\leq m$ and $1\leq j\leq n-m$.
\end{lemma}

\begin{proof} 
In order to calculate the commutation relation between $x_{ij}$ and $y_1$, we
need to consider the commutation relation between $x_{ij}$ and $M_1$. This
will be the same as the commutation relation between $x_{ij}M_1$ and $M_1$.
Set $N:=\{1,\dots,m\}\backslash\{m+1-i\}$. Then $x_{ij}M_1 = [N\cup \{j+m\}]$
and $M_1= [N\cup\{m+1-i\}]$. 
Note that $m+1-i<j+m$; so that 
$[m+1-i][j+m]=q[j+m][m+1-i]$ in $\oq(G(1,n))$. By using 
Proposition~\ref{proposition-qmuir},  
it follows that 
$M_1(x_{ij}M_1)=q(x_{ij}M_1)M_1$. Hence,  $M_1x_{ij}=qx_{ij}M_1$, and so 
$\sigma_1(x_{ij}) = qx_{ij}$ and 
$y_1x_{ij}=qx_{ij}y_1$, as claimed.\qed 
\end{proof}

\begin{lemma} \label{lemma-sigma-2}
For $1\leq i\leq m$ and $1\leq j <n-m$ 
\[
\sigma_2(x_{ij}) = qx_{ij}
\]
while $\sigma_2(x_{i,n-m}) = q^{-1}x_{i,n-m}$. Consequently, 
$y_2x_{ij}=qx_{ij}y_2$ for $1\leq i\leq m$ and $1\leq j <n-m$ 
while $y_2x_{i,n-m}=q^{-1}x_{i,n-m}y_2$.
\end{lemma} 
\begin{proof} 
When $j<n-m$, the calculations are similar to those in the proof of the 
previous result and so are omitted. 

Set $N:= \{2,\dots,m+1\}\backslash \{m+2-i\}$. Then, 
$x_{i,n-m}M_2 = [N\cup \{1\}]$
and $M_2= [N\cup\{m+2-i\}]$. 
Now, $1<m+2-i$; so that $[1][m+2-i]=q[m+2-i][1]$ in 
$\oq(G(1,n))$. By using 
Proposition~\ref{proposition-qmuir},  
it follows that 
$(x_{i,n-m}M_2)M_2=qM_2(x_{i,n-m}M_2)$. Hence, 
$x_{i,n-m}M_2=qM_2x_{i,n-m}$, and so 
$\sigma_2(x_{i,n-m}) =q^{-1}x_{i,n-m}$ and $y_2x_{i,n-m}=q^{-1}x_{i,n-m}y_2$, 
as claimed.\qed\\
\end{proof}


\section{Twisting by a $2$-cocycle} 

Given a $\mathbb{K}$-algebra $A$ that is graded by a semigroup, one can twist the
multiplication in $A$ by using a cocycle to produce a new multiplication. We
only need to deal with  $\mz^n$-graded algebras; so restrict our
discussion to this case. 

\begin{definition}
{\rm A {\em $2$-cocycle (with values in $\mathbb{K}^*$)} on $\mz^n$ is a map
$c:\mz^n\times\mz^n\goesto\mathbb{K}^*$ such that 
\[
c(s,t+u)c(t,u)=c(s,t)c(s+t,u)
\]
for all $s,t,u\in\mz^n$. 
}
\end{definition}

\q
Given a $\mz^n$-graded $\mathbb{K}$-algebra $A$ if $a$ is a homogeneous element 
in $A_s$, for $s\in\mz^n$, then we set $\content{a}:=s$. 

\q
Given a $\mz^n$-graded $\mathbb{K}$-algebra $A$ and a $2$-cocycle $c$ on $\mz^n$, 
one can define a new $\mathbb{K}$-algebra $T(A)$ in the following way. 
As a graded vector space, $A$ and $T(A)$ are isomorphic via an isomorphism 
$a\mapsto a'$. The multiplication in $T(A)$ is given by
\[
a'b':=c(s,t)(ab)'\]
for homogeneous elements $a,b\in A$ with content $s$ and $t$, respectively.
The defining condition of a $2$-cocycle is precisely the condition needed to 
ensure that this multiplication is associative. 
We refer to $T(A)$ as the {\em twist of $A$ by $c$}, and the map $a\mapsto a'$
is the {\em twist map}. \\

\q The property of being an integral domain is preserved under twists, as the 
next lemma shows. 

\begin{lemma}\label{lemma-twist-domain} 
Let $A$ be a $\mz^n$-graded $\k$-algebra that is an integral domain, and let
$c$ be a $2$-cocycle on $\mz^n$. Then $T(A)$ is an integral domain. 
\end{lemma}

\begin{proof} 
We may view $A$ as graded by $\mz^n$, which
can be made into a totally ordered group; then $T(A)$ is graded by the
same totally ordered group. In order to see that the product of two nonzero
elements $a',b'$ of $T(A)$ is nonzero, it suffices to show that the
product of their highest terms is nonzero. Hence, we may assume that $a,b$ 
are homogeneous elements. In this case, $a'b'$ is a nonzero scalar multiple 
of $(ab)'$ and $ab\neq 0$, since $A$ is a domain. Hence, 
$T(A)$ is a domain, as required.\qed
\end{proof}

\q 
Our aim is to twist the quantum grassmannian $\oqgmn$ by a suitable
$2$-cocycle in such a way that the effect of the twist map is to cycle the
indices of the generating quantum minors. There is a technical problem
associated with this attempt, in that the defining relations for the quantum
grassmannian (quantum Pl\"ucker relations) are complicated to deal with. We 
avoid the problem by using the notion of noncommutative dehomogenisation 
introduced earlier.\\ 

\q 
Let the standard basis of $\mz^n$ be denoted by 
$\{\epsilon(1),\dots,\epsilon(n)\}$,
and let $(s_1,\dots,s_n)$ denote the element $s_1\epsilon(1)+\dots+
s_n\epsilon(n)$.

\q 
The quantum grassmannian $\oqgmn$ has a natural grading by $\mz^n$ determined 
by the {\em content} of a generating quantum minor, where 
$\content{[I]}:= \sum_{i\in I} \epsilon(i)$. 

\q Note that $M_{\alpha}$ is a homogeneous element of $\oqgmn$ and so the
$\mz^n$-grading of $\oqgmn$ extends in a natural way to
$\oqgmn[[M_{\alpha}]^{-1}]$ and hence to $\mathbb{K}[x_{ij}][y_{\alpha}^{\pm
1};\sigma_{\alpha}]$ by using the dehomogenisation isomorphism of
Theorem~\ref{theorem-dehomiso}.

\begin{lemma} Let $p=q^{2/m}$. 
The map  $c:\mz^n\times\mz^n\goesto \mathbb{K}^*$ defined by
\[
c((s_1,\dots,s_n),(t_1,\dots,t_n)):=\prod_{j\neq n}\,p^{s_nt_j}.
\]
is a $2$-cocycle.
\end{lemma}

\begin{proof} Set $s=(s_1,\dots,s_n)$, $t=(t_1,\dots,t_n)$ and 
$u=(u_1,\dots,u_n)$. We have to check that 
\[
c(s,t+u)c(t,u)
=c(s,t)
c(s+t,u).
\]
The proof is routine, one checks that each side is equal to 
\[
\prod_{j\neq n}\,p^{s_n t_j+ s_n u_j+ t_n u_j}.
\]\qed 
\end{proof}
~\\

\q 
Now, we look at the effect of twisting the algebra 
$A :=\mathbb{K}[x_{ij}][y_1^{\pm 1};\sigma_1]$
by using the $2$-cocycle $c$. Write $y$ and $\sigma$ for $y_1$ and $\sigma_1$,
respectively.\\

\q 
We denote by $T(A)$ the twist of $A$ by using the 
$2$-cocyle $c$; so that if $a,b$ are homogeneous elements with content
$s=(s_1,\dots,s_n)$ and $t=(t_1,\dots,t_n)$, respectively, then 
\[
a'b':=c((s_1,\dots,s_n),(t_1,\dots,t_n))(ab)'.
\]

\q 
Now, we are in the case that $\alpha=1$, so that 
\[
x_{ij}=[\{1,\dots,m\}\cup \{j+m\}\backslash \{m+1-i\}][1,\dots,m]^{-1}.
\] 
Note that the content of $x_{ij}$ is $\epsilon(j+m)-\epsilon(m+1-i)$
and that the content of $y$ is $\epsilon(1)+\dots+\epsilon(m)$. 

\q 
As $A$ is generated by the homogeneous elements $x_{ij}$ and $y$, 
the twisted algebra $T(A)$ is generated by the homogeneous elements 
$x'_{ij}$ and $y'$. Our first aim is to describe the commutation 
relations satisfied by these elements.

\q
We will often abuse notation by writing $c(a,b)$ instead of 
$c(\content{a},\content{b})$ for homogeneous elements $a,b\in A$.

\q 
Note that the value taken by $c$ on a pair of elements from the set 
$\{x_{ij}, y\}$ is often equal to $p^0=1$. In fact, the only possibilities 
for a value other than $p^0$ occur in the  cases when $\epsilon(n)$ occurs 
in the content of the first argument in $c$. This can only occur for 
$x_{i,n-m}$ and we check that 
\[
c(x_{i,n-m},x_{l,n-m})=p^{-1},
\qquad\qquad c(x_{i,n-m},y)=p^m =q^2.
\]
while $c(x_{i,n-m},x_{l,j})= 1$ for $j<n-m$ and $c(y,x_{ij})=1$ for 
all $i,j$. 
These observations make the calculation of the twisted product on pairs 
from the set $\{x_{ij}', y'\}$ very easy. \\

\begin{lemma}\label{theorem-twist} 
$(x_{ij}')$ is a generic $q$-quantum matrix; that is, 
the algebra $\mathbb{K}[x'_{ij}]$
is isomorphic to $\mathcal{O}_q(M_{m,n-m})$.
Moreover 
\[
y'x_{ij}'=qx_{ij}'y'{\rm ~~for~~}j<n-m,\quad{\rm ~~and~~}\quad 
y'x_{i,n-m}'=q^{-1}x_{i,n-m}'y'.
\]
\end{lemma}

\begin{proof} 
First, we show that the $x_{ij}'$ satisfy the commutation relations for  a 
$q$-quantum matrix. The cases where 
$c(-,-)$ takes value $1$ are easy to check, for example, 
for $i_1<i_2$ and $j<n-m$,
\[
x_{i_1j}'x_{i_2j}'
= c(x_{i_1j},x_{i_2j})(x_{i_1j}x_{i_2j})'=(x_{i_1j}x_{i_2j})'
\]
while 
\[
x_{i_2j}'x_{i_1j}'= c(x_{i_2j},x_{i_1j})(x_{i_2j}x_{i_1j})'
=(x_{i_2j}x_{i_1j})'= q^{-1}(x_{i_1j}x_{i_2j})' 
=q^{-1}x_{i_1j}'x_{i_2j}'
\]
and so $x_{i_1j}'x_{i_2j}'=qx_{i_2j}'x_{i_1j}'$, as required. 

\q 
Also, for $i_1<i_2$, 
\[
x_{i_1(n-m)}'x_{i_2(n-m)}'= p^{-1}(x_{i_1(n-m)}x_{i_2(n-m)})'
\]
and
\[
x_{i_2(n-m)}'x_{i_1(n-m)}'= p^{-1}(x_{i_2(n-m)}x_{i_1(n-m)})'
\]
so again the desired $q$-commutation follows and the column relations 
are established. 

\q 
The row computations are similar and so are omitted. 

\q
When $i_1<i_2$ and $j_1<j_2$, note that 
$c(x_{i_1j_2},x_{i_2j_1}) = c(x_{i_2j_1},x_{i_1j_2})=1$; and so 
\[
x_{i_1j_2}'x_{i_2j_1}'= (x_{i_1j_2}x_{i_2j_1})'
=
(x_{i_2j_1}x_{i_1j_2})'=x_{i_2j_1}'x_{i_1j_2}',
\]
as required. 

\q 
Continuing with $i_1<i_2$ and $j_1<j_2$, note that 
$c(x_{i_1j_1},x_{i_2j_2}) = c(x_{i_2j_2},x_{i_1j_1})=1$; and so 
\begin{eqnarray*}
x_{i_1j_1}'x_{i_2j_2}'- x_{i_2j_2}'x_{i_1j_1}'
&=&
(x_{i_1j_1}x_{i_2j_2})'- (x_{i_2j_2}x_{i_1j_1})'
=
(x_{i_1j_1}x_{i_2j_2} - x_{i_2j_2}x_{i_1j_1})'\\
&=&
(q-q^{-1})(x_{i_1j_2}x_{i_2j_1})'
=
(q-q^{-1})x_{i_1j_2}'x_{i_2j_1}'.
\end{eqnarray*}
This finishes the verification that the $x_{ij}'$ satisfy the commutation 
relations of $\mathcal{O}_q(M_{m,n-m})$. 
As a result, there is an epimorphism from $\mathcal{O}_q(M_{m,n-m})$ onto 
$\k[x_{ij}']$. If this epimorphism were not an isomorphism then 
$\gk(\k[x_{ij}'])<\gk(\mathcal{O}_q(M_{m,n-m}))=m(n-m)$, by 
\cite[Proposition 3.15]{kl}, since $\mathcal{O}_q(M_{m,n-m})$ is a domain. 

\q 
However, any monomial $x_{i_{1}j_{1}}'x_{i_{2}j_{2}}'\dots x_{i_{t}j_{t}}'$ is
a nonzero scalar multiple of $(x_{i_{1}j_{1}}x_{i_{2}j_{2}}\dots
x_{i_{t}j_{t}})'$; and so a linear combination of such monomials is zero if and only if a corresponding linear combination of monomials in the 
$x_{ij}$ is zero. It follows that $\gk(\k[x_{ij}'])=\gk(\k[x_{ij}])
=m(n-m)$. Thus, $\k[x_{ij}']\cong\mathcal{O}_q(M_{m,n-m})$.

\q 
Now, we calculate how $y'$ commutes with the $x_{ij}'$. 

\q 
For $j<n-m$, observe that 
\[
x_{ij}'y'=c(x_{ij},y)(x_{ij}y)'=(x_{ij}y)'
\]
and so
\[
y'x_{ij}'=c(y,x_{ij})(yx_{ij})'=(yx_{ij})'=q(x_{ij}y)'
=qx_{ij}'y'.
\]
Finally, 
\[
x_{i,n-m}'y'= c(x_{i,n-m},y)(x_{i,n-m}y)'=q^{2}(x_{i,n-m}y)'
\]
and so 
\[
y'x_{i,n-m}'= c(y,x_{i,n-m})(yx_{i,n-m})'
=(yx_{i,n-m})'=q(x_{i,n-m}y)' = q^{-1}x_{i,n-m}'y'.
\]
\qed 
\end{proof}

\q
We now wish to consider the dehomogenisation isomorphism when 
$\alpha=2$. In order to avoid a clash of notation, we will write 
\[
\oqgmn[[M_2]^{-1}]\cong \mathbb{K}[z_{ij}][w^{\pm 1};\phi]
\]
where $z_{ij} 
:= [M_2\cup\{\widetilde{j+m+1}\}\backslash \{m+2-i\}]$ 
and $M_2= [2,3,\dots,m+1]$.

\begin{theorem}
\[
T(\mathbb{K}[x_{ij}][y^{\pm 1};\sigma])\quad
\cong \quad 
\mathbb{K}[z_{ij}][w^{\pm 1};\phi]
\]
via a map $\theta:T(\mathbb{K}[x_{ij}][y^{\pm 1};\sigma])\goesto 
\mathbb{K}[z_{ij}][w^{\pm 1};\phi]$ that sends 
$x_{ij}'$ to $z_{ij}$ 
and $y$ to $w$.
\end{theorem}

\begin{proof} From Lemma~\ref{lemma-sigma-2} and 
Lemma~\ref{theorem-twist}, we 
see that the 
commutation relations among the $\{x_{ij}',y_1'\}$ of 
$T(\mathbb{K}[x_{ij}][y_1^{\pm 1};\sigma_1])$ are the same as the corresponding 
commutation relations among the generating set
$\{z_{ij}',y_{2}'\}$ of 
$\mathbb{K}[z_{ij}][y_2^{\pm 1};\sigma_2]$. 

\q 
Thus, we may define a homomorphism from $T(\mathbb{K}[x_{ij}][y^{\pm 1};\sigma])$ to
$\mathbb{K}[z_{ij}][w^{\pm 1};\phi]$ by sending $x_{ij}'$ to $z_{ij}'$ and $y$ to $w$.
This homomorphism is an epimorphism, since the generators of $\mathbb{K}[z_{ij}][w^{\pm
1};\phi]$ are in the image. Finally, the two algebras have the same 
Gelfand-Kirillov dimension, $m(n-m)+1$; 
so this epimorphism between two domains 
must also be a monomorphism, by \cite[Proposition 3.15]{kl}. 
\qed
\end{proof}

\q
We may identify 
$\oqgmn$ as a subalgebra of $\mathbb{K}[x_{ij}][y_1^{\pm 1};\sigma_1]$ 
via the dehomogenisation isomorphism 
$\oqgmn[[M_1]^{-1}]\cong \mathbb{K}[x_{ij}][y_1^{\pm 1};\sigma_1]$ and 
identify another copy of $\oqgmn$ with a 
subalgebra of $\mathbb{K}[z_{ij}][y_2^{\pm 1};\sigma_2]$ via the 
isomorphism 
$\oqgmn[[M_2]^{-1}]\cong \mathbb{K}[z_{ij}][y_2^{\pm 1};\sigma_2]$. Our next aim is 
to show that the image of the first copy of $\oqgmn$ under the map 
$\theta\circ T$ is the second copy of $\oqgmn$. In order to do this, we 
need to track the image of a generating quantum minor through the sequence 
of maps
\[
\oqgmn\stackrel{\phi_1}\goesto \mathbb{K}[x_{ij}][y_1^{\pm 1};\sigma_1]
\stackrel{T}\goesto \mathbb{K}[x_{ij}'][y_1'^{\pm 1}]
\stackrel{\theta}\goesto  \mathbb{K}[z_{ij}][y_2^{\pm 1};\sigma_2]
\stackrel{\rho_2}\goesto \oqgmn[[M_2]^{-1}] 
\]

\q 
First, we record the effect of the twist map on quantum minors. We need to 
consider quantum minors in each of the quantum matrix algebras 
$\mathbb{K}[x_{ij}]$ and  $\mathbb{K}[x_{ij}']$; so for a given row set $I$ and 
column set $J$ we will denote the corresponding quantum minors by 
$[I|J]_x$ and $[I|J]_{x'}$, respectively. 

\begin{lemma} Let $[I|J]_x$ be a quantum minor of the quantum matrix algebra $\mathbb{K}[x_{ij}]$ 
in the previous theorem.
Then the image of $[I|J]_x$ under the twist map is $[I|J]_{x'}$.
\end{lemma}

\begin{proof}
This proof is a routine calculation, using induction on the size of the 
quantum minor and quantum Laplace expansions, noting that 
each $c(-,-)$ that occurs takes value 1.
\qed\\
\end{proof}

\begin{lemma} 
$c([I|J],y) = 1$ when $n-m\not\in J$ and $c([I|J],y) = q^2$ 
when $n-m\in J$ 
\end{lemma}

\begin{proof}
This follows from the fact that $\epsilon(n)$ appears (with nonzero coefficients) in $\content{[I|J]}$ if and only if $n-m\in J$ by \cite[Proposition 4.3]{len-russ}. \qed \\
\end{proof}

\q 
As before, for a given row set $I$ and 
column set $J$ we will denote the corresponding quantum minors of the various quantum matrix algebras by 
$[I|J]_x, [I|J]_{x'}$ and $[I|J]_z$, respectively. 

\begin{lemma}
Let $I=[i_1,\dots,i_m]$ be a generating quantum minor of $\oqgmn$. Then 

\[
\rho_2\circ\theta\circ T \circ\phi_1([I]) = 
\left\{
\begin{array}{l}~~~~[i_1+1,\dots,i_m+1]\quad 
{\rm if~}i_m\neq n\\
q^{-2}[1, i_1+1,\dots,i_{m-1}+1]
\quad {\rm ~~~~if~}i_m= n
\end{array}
\right.
\]
\end{lemma} 

\begin{proof}Note that 
\[
\phi_1(I)= [(m+1)-M_1\backslash I_r | I_c -m]_xy_1;
\]
and so
\[
T\circ\phi_1(I) = ([(m+1)-M_1\backslash I_r | I_c -m]_xy_1)'
= C^{-1}[(m+1)M_1\backslash I_r | I_c -m]_{x'}y_1'
\]
where $C:= c([(m+1)-M_1\backslash I_r | I_c -m]_x,y)$ and note that 
$C=1$ if $n-m\not\in I_c -m$ (and so if $n\not\in I$), while 
$C=q^2$ if $n-m\in I_c -m$ (and so if $n\in I$).

\q
Thus, 
\begin{eqnarray*}
\theta\circ T\circ\phi_1(I) &=& 
C^{-1}[(m+1) -M_1\backslash I_r | I_c -m]_{z}w \\
&=&
C^{-1}[(m+2) -M_2\backslash (I_r+1) | (I_c+1) -(m+1)]_{z}w 
\end{eqnarray*}
Finally, 
\begin{eqnarray*}
\rho_2\circ\theta\circ T\circ\phi_1(I)
&=&
C^{-1}\rho_2([(m+2) -M_2\backslash (I_r+1) | (I_c+1) -(m+1)]_{z}w)\\
&=&
C^{-1}[(I_r+1)\sqcup  (I_c+1)] = C^{-1}[I+1]
\end{eqnarray*}
and the result follows. Note that the last equality is obtained by 
the same calculation as in the proof of Corollary~\ref{corollary-phi-of-I}.
\qed
\end{proof}

\q
We can now reach our conclusion.

\begin{theorem} \label{theorem-T-iso}
\[
T(\oqgmn)\cong\oqgmn
\]
via a map $\theta$ that sends 
$[i_1,\dots,i_m]'$ to $[i_1+1,\dots,i_m+1]$, for $i_m< n$, and 
$[i_1,\dots,i_{m-1},n]$ is sent to $q^{-2}[1, i_1+1,\dots,i_{m-1}+1]$. 
\end{theorem}

\begin{proof}This follows immediately from the previous lemma. \qed
\end{proof}

\section{Twisting the $\ch$-prime spectrum}

In this section we assume that $q$ is a not a 
root of unity, in order that we
know that the prime ideals of $\oqgmn$ 
are completely prime, see \cite[Theorem
5.2]{llr}. 

\q The natural $\mz^n$-grading on $\oqgmn$ induces a rational action of the
algebraic torus $\ch:=(\mathbb{K}^*)^n$ on $\oqgmn$ by
$\mathbb{K}$-automorphisms via 
\[ (h_1 ,\dots , h_n ). [i_1, \dots , i_m]
= h_{i_1} \cdots h_{i_m}[i_1, \dots , i_m],
\]
(see \cite[Lemma II.2.11]{bg} for more details). In this setting, the
homogeneous prime ideals of $\oqgmn$ are exactly those primes that are
invariant under this torus action. Hence homogeneous primes are also called
{\it $\ch$-primes}, and the set $\hspec (\oqgmn)$ of all $\ch$-primes of
$\oqgmn$ is called the {\it $\ch$-prime spectrum} of $\oqgmn$. It was proved
in \cite{llr} that this set is finite, and its cardinality was computed. The
importance of the $\ch$-prime spectrum was pointed out by Goodearl and Letzter
who proved that the $\ch$-prime spectrum parametrizes a natural stratification
of the prime spectrum of $\oqgmn$.

\begin{theorem} 
Suppose that $q$ is not a root of unity. 
Let $P$ be an $\ch$-prime ideal of $\oqgmn$. Then 
$T(P):=\{p'\mid p\in P\}$ is an $\ch$-prime ideal of $T(\oqgmn)$. 
\end{theorem}

\begin{proof}
The algebra $\oqgmn/P$ inherits a $\mz^n$-grading, as $P$ is homogeneous; and
so we can form the twisted algebra $T(\oqgmn/P)$. It then follows that
$T(\oqgmn/P)\cong T(\oqgmn)/T(P)$. Hence, it is enough to show that
$T(\oqgmn/P)$ is a domain and this follows from
Lemma~\ref{lemma-twist-domain}. 
\end{proof}

\begin{corollary} \label{corollary-primes-permuted}
Suppose that $q$ is not a root of unity. Then 
\[
\theta(T(\hspec(\oqgmn)))= \hspec(\oqgmn),
\] 
where $\theta$ is 
the isomorphism defined in Theorem~\ref{theorem-T-iso}.
\end{corollary}

\begin{proof}
If $P,Q$ are two distinct $\ch$-prime ideals of $\oqgmn$ then $T(P)$ and 
$T(Q)$ are distinct $\ch$-prime ideals of $T(\oqgmn)$; and so their images 
under the isomorphism $\theta$ are distinct $\ch$-prime ideals of 
$\oqgmn$. As the set of $\ch$-prime ideals is finite, this establishes 
the claim.\qed \\
\end{proof}

\q
It follows that if $P$ is an $\ch$-prime ideal of $\oqgmn$ then a quantum
minor $[i_1,\dots,i_m]$ is in $P$ if and only if the quantum minor 
$[i_1+1,\dots,i_m+1]$ is in $\theta(T(P))$, where $i_m+1:=1$ if $i_m=n$. 
In other words, the sets of quantum minors that are in $\ch$-prime ideals 
are permuted by $\theta\circ T$. 

\q
Note that in \cite{llr}, it was 
shown that each $\ch$-prime ideal of $\oqgtwofour$ 
is generated by the quantum minors that it contains, and it was conjectured 
that this holds in any $\oqgmn$.


\newpage 

\noindent 
S Launois: \\
School of Mathematics, Statistics and Actuarial Science,\\
University of Kent\\
Canterbury, Kent CT2 7NF, UK\\
Email: {\tt S.Launois@kent.ac.uk} \\

\noindent 
T H Lenagan: \\
Maxwell Institute for Mathematical Sciences\\
School of Mathematics, University of Edinburgh,\\
James Clerk Maxwell Building, King's Buildings, Mayfield Road,\\
Edinburgh EH9 3JZ, Scotland, UK\\
E-mail: {\tt tom@maths.ed.ac.uk} 



\begin{thebibliography}{99}

\bibitem{bg} K A Brown and K R Goodearl, \emph{Lectures on Algebraic Quantum Groups},
  Advanced Courses in Mathematics. CRM Barcelona, Birkh\"auser Verlag, Basel, 2002.

\bibitem{good-laun-len} K R Goodearl, S Launois and T H Lenagan, 
{\em Totally nonnegative cells and matrix Poisson varieties}, 
arXiv:0905.3631.

\bibitem{good-laun-len2} K R Goodearl, S Launois and T H Lenagan
{\em Torus-invariant prime ideals in quantum matrices, 
totally nonnegative cells and symplectic leaves}, arXiv:0909.3935.

\bibitem{good-len} K R Goodearl and T H Lenagan, {\em Quantum determinantal 
ideals}, Duke Math. J. 103 (2000) 165-190.

\bibitem{good-yak} K R Goodearl and M Yakimov, 
{\em Poisson structures on affine spaces and flag varieties. II}, 
Trans. Amer. Math. Soc. 361  (2009), 5753-5780

\bibitem{klr}
A Kelly, T H Lenagan, and L Rigal, \emph{Ring theoretic properties of
  quantum grassmannians}, J. Algebra Appl. \textbf{3} (2004), no.~1, 9--30.

\bibitem{kl} G R Krause and T H Lenagan, {\em Growth of algebras and 
Gelfand-Kirillov dimension}, 
Graduate Studies in Mathematics, 22. American Mathematical Society, Providence, RI, 2000

\bibitem{kroblec}
D~Krob and B~Leclerc, \emph{Minor identities for quasi-determinants and
  quantum determinants}, Comm. Math. Phys. \textbf{169} (1995), no.~1, 1--23.

\bibitem{knut-lam-speyer} A Knutson, T Lam and D E Speyer, 
 {\em Positroid varieties I: juggling and geometry}, arXiv:0903.3694.

\bibitem{llr} S Launois, T H Lenagan and L Rigal, \emph{Prime ideals in 
the quantum grassmannian}, Selecta Mathematica 13 (2008), 697--725. 

\bibitem{lr2}
T H Lenagan and L Rigal, \emph{Quantum graded algebras with a straightening
  law and the {AS}-{C}ohen-{M}acaulay property for quantum determinantal rings
  and quantum grassmannians}, J. Algebra \textbf{301} (2006), no.~2, 670--702.

\bibitem{lr3}
T H Lenagan and L Rigal, \emph{Quantum analogues of {S}chubert varieties in the grassmannian},
  Glasgow Math. J. \textbf{50} (2008), no.~1, 55--70.

\bibitem{len-russ} T H Lenagan and E J Russell, {\em Cyclic orders 
on the quantum grassmannian}, Arabian Journal for Science and Engineering {\bf 33} (2008), 337--350. 

\bibitem{post} A~Postnikov,~{\em Total~positivity,~Grassmannians,~and~networks}, arXiv:0609764.

\bibitem{yak} M Yakimov, {\em Cyclicity of Lusztig's stratification of
grassmannians and Poisson geometry}, arXiv:0902.2181.

\end{thebibliography}
\end{document}